\documentclass[12pt]{amsart}
\usepackage{amssymb}
\usepackage{graphics}  
\usepackage{amsmath}
\usepackage{enumerate}
\usepackage{hyperref}

\newtheorem{theorem}{Theorem}[section]
\newtheorem{lemma}[theorem]{Lemma}

\newtheorem{corollary}[theorem]{Corollary}

\newtheorem{remark}[theorem]{Remark}
\newtheorem{example}[theorem]{Example}
\newtheorem{definition}[theorem]{Definition}

\def\ord{\mathrm{ord}\,}
\def\a{\mathbf{a}}
\def\b{\mathbf{b}}
\def\per{\rho}
\def\lcm{\mathrm{lcm}}
\def\proj{\pi}
\def\Zpos{\mathbb{Z}_{>0}}
\begin{document}
\title[Period sets of linear recurrences over finite fields and related commutative rings]
{Period sets of linear recurrences over finite fields and related commutative rings} 
\author{Michael R. Bush}
\address{Department of Mathematics,
  Washington and Lee University, 204 W. Washington St., Lexington, VA 24450, USA.}
\email{bushm@wlu.edu}
\author{Danjoseph Quijada}
\address{Department of Mathematics, University of Southern California, 3620 S. Vermont Ave., KAP 104, Los Angeles, CA 90089-2532, USA.}
\email{dquijada@usc.edu}
\date{8 May 2018}
\thanks{This work began as a summer project while the second author was an undergraduate student at Washington and Lee University. It subsequently became part of an Honors thesis~\cite{Q} completed in 2015. The second author would like to thank the university for providing summer research scholar funding. The first author was also partially supported by an internal Lenfest grant.}


\begin{abstract}
After giving an overview of the existing theory regarding the periods of sequences defined by linear recurrences over finite fields, we give explicit descriptions of the sets of periods that arise if one considers all sequences over $\mathbb{F}_q$ generated by linear recurrences for a fixed choice of the degree $k$ in the range $1 \leq k \leq 4$. We also investigate the periods of sequences generated by linear recurrences over rings of the form $\mathbb{F}_{q_1} \oplus \ldots \oplus \mathbb{F}_{q_r}$. 
\end{abstract}

\maketitle

%%%%%%%%%%%%%%%%%%%%%%%%%%%%%%%%%%%%%%%%%%%%%%%%%%%%%%%%%%%%%%%%%%%%%%%%%%%%%%%%%%
%%%%%%%%%%%%%%%%%%%%%%%%%%%%%%%%%%%%%%%%%%%%%%%%%%%%%%%%%%%%%%%%%%%%%%%%%%%%%%%%%%
\section{Introduction}\label{section-introduction}
We begin by fixing some notation and recalling some definitions and basic results concerning linear recurrence sequences. Let $R$ be a commutative ring with unity $1 \neq 0$. Let $\mathbf{a} = (a_n)_{n=0}^\infty$ be a sequence with $a_n \in R$ for all $n \geq 0$. If there exist $c_i \in R$ for $0 \leq i \leq k-1$ such that the terms in the sequence satisfy the following equation
\begin{equation}\label{eqn-lin-rec}
  a_{n+k}  = \sum_{i=0}^{k-1} c_i a_{n+i} 
  \end{equation}
for all $n \geq 0$, then we say that $\mathbf{a}$ is a {\em linear recurrence sequence over $R$}. The equation is called a {\em linear recurrence} and the quantity $k$ is called the {\em degree} of the recurrence. Formally, the linear recurrence is specified by a tuple of coefficients $(c_0, \ldots, c_{k-1}) \in R^k$ and zero entries are allowed in any position including $c_0$, thus the degree is not required to be minimal in any sense. 
\begin{example}
The well-known Fibonacci sequence
\[ 0,1,1,2,3,5,8,13,\ldots \]
is a linear recurrence sequence over $\mathbb{Z}$ satisfying the degree~$2$ recurrence
\[ a_{n+2} = a_{n+1} + a_n \]
for all $n \geq 0$ with coefficients $c_0 = c_1 = 1$.
\end{example}

It is clear from the form of equation~(\ref{eqn-lin-rec}) that each term in the sequence depends on the previous $k$ terms. As a result, it is natural to consider  the tuple  $s_n = (a_n, \ldots, a_{n + k - 1}) \in R^k$ ($n \geq 0$), which we refer to as the {\em $n$-th state vector} of the sequence.  Any linear recurrence sequence of degree $k$ is completely determined by specifying the recurrence equation and initial state $s_0 = (a_0,\ldots,a_{k-1})$. 

If $R$ is a finite ring, then there are a finite number of possible states $|R^k| = |R|^k$ and an argument using the pigeonhole principle shows that the sequence must be ultimately periodic, ie. there exist positive $N,m \in \mathbb{Z}$ such that $a_{n+m} = a_n$ for all $n \geq N$. In this situation, we will say the sequence $\a$ is {\em $m$-periodic}. The smallest positive value of $m$ with this property is called the {\em period} of the sequence and it will be denoted by $\per(\mathbf{a})$. We note that some authors use the term {\em least period} for $\per(\a)$, but we will not. It is straightforward to verify using a Division Algorithm argument that for any positive integer $m$, the sequence $\mathbf{a}$ is $m$-periodic if and only if $\per(\mathbf{a})$ divides $m$. 

Important examples of finite commutative rings include the integers modulo~$n$ which we denote $\mathbb{Z}_n$ and also finite fields.  We use the notation $\mathbb{F}_q$ to denote the finite field of order $q$ where $q = p^e$ with $e \geq 1$ and  $p$  a prime integer. Recall that when $n = q = p$ these notions coincide and we have $\mathbb{Z}_p \cong \mathbb{F}_p$.  

\begin{example}\label{example-fib-mod-2}
Consider the Fibonacci recurrence $a_{n+2} = a_{n+1} + a_n$ over $\mathbb{F}_2  = \{0 , 1\}$. Starting from the initial state $s_0 = (a_0,a_1) = (0,1)$, we obtain the sequence 
\[ \underline{0, 1}, 1, \underline{0, 1}, 1, 0, 1, \ldots \]
and so clearly we have $\per(\a) = 3$ in this case. 
\end{example}

In general, if the coefficient $c_0 \in R$ is a unit, then one can invert the recurrence equation~(\ref{eqn-lin-rec}) and observe that any state $s_n$ uniquely determines the preceding term $a_{n-1}$ and hence the state $s_{n-1}$. In this situation, the periodicity begins right from the initial term. ie. $a_{n+m} = a_n$ for all $n \geq 0$. Throughout this paper we will restrict attention to sequences defined by recurrence equations satisfying this condition.

It is natural to ask about the set of all possible periods of linear recurrence sequences defined over a finite ring $R$ and how this depends on the choice of ring. One easy observation is that for a given degree $k$, the period is always bounded above by $|R|^k - 1$. This follows from the pigeonhole principle argument above. Since there are $|R|^k$ state vectors, any list of $|R|^k + 1$ consecutive states must include the repetition of at least two states. This gives a bound of $|R|^k$ on the period. One can decrease this slightly because the zero state vector $(0,\ldots,0)$ clearly generates the zero sequence of period $1$. Hence any sequence containing infinitely many non-zero terms cannot involve this state and so the maximum number of allowed states in a repeating cycle is reduced to $|R|^k - 1$.

In the case where $R = \mathbb{F}_q$, there is a well-developed theory describing the behavior of these periods. In particular, it can be shown (see Corollary~\ref{corollary-max-period-finite-field}) that for each $k \geq 1$, there exist linear recurrence sequences of degree $k$ with period equal to  the upper bound $q^k - 1$. In Section~\ref{section-lin-rec-ff}, we give an overview of some of this theory. In Section~3, we apply it to obtain our main result (Theorem~\ref{theorem-period-sets}) giving explicit descriptions of the sets of periods that can be achieved for small values of the degree $k$ over an arbitrary finite field. In Section~4, we investigate the periods of sequences defined over a slightly broader class of finite commutative rings.

Throughout we will assume that the reader is familiar with the basic facts about finite fields and polynomial rings defined over fields as can be found in many introductory books on abstract algebra. See for instance~\cite{G, J}.

%%%%%%%%%%%%%%%%%%%%%%%%%%%%%%%%%%%%%%%%%%%%%%%%%%%%%%%%%%%%%%%%%%%%%%%%%%%%%%%%%%
%%%%%%%%%%%%%%%%%%%%%%%%%%%%%%%%%%%%%%%%%%%%%%%%%%%%%%%%%%%%%%%%%%%%%%%%%%%%%%%%%%
\section{Linear recurrences over finite fields}\label{section-lin-rec-ff}

In this section, we review some material concerning orders of polynomials and then explain how this is connected with determining the periods of linear recurrence sequences defined over finite fields. We follow the treatment in~\cite[Chapters~3 and 8]{LN} and further discussion and additional results can be found there. See also~\cite[Section~6.2]{B} although we note that the latter uses slightly different terminology.

\begin{definition}
Let $f(x) \in \mathbb{F}_q[x]$ and suppose $f(x) = x^r g(x)$ with $r \geq 0$ and $\gcd(g(x),x) = 1$. Observe that $r$ and $g(x)$ are uniquely determined by $f(x)$. The {\em order of $f(x)$} (denoted $\ord(f(x))$) is defined to be the smallest integer $n > 0$ such that $g(x)$ divides $x^n - 1$. Equivalently, this quantity is the multiplicative order of $x + \langle g(x) \rangle$ in the group of units of the quotient ring $\mathbb{F}_q[x]/\langle g(x) \rangle$.
\end{definition}

\begin{remark}\label{remark-ord-irred}
The fact that such an $n$ exists follows since the quotient ring $\mathbb{F}_q[x]/\langle g(x) \rangle$ is finite and so has a finite multiplicative group of units. The condition $\gcd(g(x),x) = 1$ ensures that $\alpha = x + \langle g(x) \rangle$ is an element of the unit group.

If the polynomial  $f(x)$ (or $g(x)$) is irreducible in $\mathbb{F}_q[x]$, then we have an isomorphism $\mathbb{F}_q[x]/\langle g(x) \rangle \cong \mathbb{F}_{q^d}$ were $d = \deg g(x)$. In this situation, $\alpha \neq 0$ is a root of $g(x)$ in the larger field $\mathbb{F}_{q^d}$ and the order of the polynomial is simply the multiplicative order of $\alpha$ in $\mathbb{F}_{q^d}^\times = \mathbb{F}_{q^d} - \{ 0 \}$. It follows by Lagrange's theorem that $\ord(f(x))$ must divide $q^d - 1$ in this situation.
\end{remark}

The problem of computing $\ord(f(x))$ can be reduced to the irreducible case using the following two results. For proofs, see~\cite[Theorems~3.8, 3.9 and 3.11]{LN}.

\begin{theorem}\label{theorem-ord-irred-power}
Let $g(x)$ be irreducible over $\mathbb{F}_q$ with $g(0) \neq 0$ and $\ord(g(x)) = e$, and let $f(x) = g(x)^b$ with $b \in \Zpos$. Let $t$ be the smallest integer with $p^t \geq b$ where $p$ is the characteristic of $\mathbb{F}_q$. Then $\ord(f(x)) = e p^t$.
\end{theorem}

\begin{theorem}\label{theorem-ord-prod}
Let $g_1(x), \ldots, g_k(x) \in \mathbb{F}_q[x]$ be pairwise relatively prime nonzero polynomials, and let $f(x) = g_1(x) \ldots g_k(x)$. Then $\ord(f(x))$ is equal to the least common multiple of $\ord(g_1(x)), \ldots, \ord(g_k(x))$.
\end{theorem}

In order to establish the connection between the periods of linear recurrence sequences and the orders of certain polynomials, we must first introduce the following matrix.
Using the coefficients in Equation~(\ref{eqn-lin-rec}), we define a $k \times k$ matrix $C$ over $\mathbb{F}_q$ by
\begin{equation}\label{eqn-special-matrix}
C = 
\left(
\begin{matrix}
0 & 0  & \ldots & 0 & c_0 \\
1 & 0  & \ldots & 0 & c_1 \\
0 & 1  & \ldots & 0 & c_2 \\
\vdots & \vdots & \ddots & \vdots \\
0 & 0 & \ldots & 1 & c_{k-1} 
\end{matrix}
\right).
\end{equation}
This matrix is called the {\em companion matrix} of the linear recurrence.
It is straightforward to verify that $s_n C = s_{n+1}$ for all $n \geq 0$ where $s_n$ is the $n$th state vector of a sequence satisfying the recurrence and we view $s_n$ as a row vector. By induction, it follows that $s_n =  s_0 C^n$ for all $n \geq 0$. From this point onwards, we will assume that $\det C = (-1)^{k-1} c_0 \neq 0$ and hence that $C$ is an element of the general linear group $GL_k(\mathbb{F}_q)$. Recall from Section~\ref{section-introduction} that under this assumption on $c_0$, all sequences associated with the recurrence are periodic starting from the initial state. 

There is an important relationship between the periods of sequences defined by a given linear recurrence and the order of the corresponding companion matrix $C$ in the multiplicative group $GL_k(\mathbb{F}_q)$. Recall that the order of a matrix $C$ is the smallest positive integer $n$ such that $C^n = I_k$ where $I_k \in GL_k(\mathbb{F}_q)$ is the $k \times k$ identity matrix. We will  denote this quantity by $\ord(C)$.

\begin{theorem}\label{theorem-period-matrixorder}
Let $\mathbf{a}$ be a sequence satisfying a linear recurrence and let $C$ be the companion matrix, then $\per(\mathbf{a})$ divides $\ord(C)$. Furthermore, for any given companion matrix $C$, there exists a sequence $\b$ satisfying the associated recurrence with $\per(\b) = \ord(C)$.
\end{theorem}
\begin{proof}
This result appears in~\cite[Theorems~8.13 and 8.17]{LN}. Since the proof is fairly short and contains some important ideas, we recall it here. 

Let $\a$ be an arbitrary sequence satisfying the given recurrence. We first show that $\per(\a)$ divides $\per(\b)$ for a special sequence $\b$ which we now introduce.  Let $\b$ be the sequence satisfying the same recurrence with initial state $t_0 = (0,\ldots,0,1)$ -- so the first $k-1$ entries equal $0$ and the $k$th entry is $1$. Then the $n$th state $t_n$ satisfies $t_n = t_0 C^n$ for all $n \geq 0$. The first $k$ states then have the following special form
\begin{eqnarray*}
t_0 &=& (0,\ldots,0,0,1) \\
t_1 &=& (0,\ldots,0,1,*) \\ 
t_2 &=& (0, \ldots, 1, *, *) \\ 
\vdots && \qquad \vdots \\
t_{k-1} &=& (1,*, \ldots, *, *, *)
\end{eqnarray*} 
where  $*$ denotes some element in $\mathbb{F}_q$ that we do not specify exactly. Clearly, $\mathcal{B} = \{t_i \mid 0 \leq i \leq k-1 \}$ is a basis for $\mathbb{F}_q^k$ since the matrix formed by these vectors has full rank~$k$. It follows that the initial state $s_0$ of $\a$ can be expressed as a linear combination of the vectors in $\mathcal{B}$. Using the linearity of the recurrence, one can then see that the sequence $\a$ is a linear combination of the sequences with initial states in $\mathcal{B}$. These sequences consist of $\b$ together with its cyclic shifts and so all have the same period $\per(\b)$. Any linear combination of $m$-periodic sequences is $m$-periodic, so the sequence $\a$ must be $\per(\b)$-periodic. It follows that $\per(\a)$ divides $\per(\b)$.

To finish, we now show that $\per(\b) = \ord(C)$. Let $m = \per(\b)$ and $n = \ord(C)$. Since $C^n = I_k$, we see that $t_n = t_0 C^n = t_0$. It follows that $\b$ is $n$-periodic and so $m$ must divide $n$. We note that this part of the argument does not make use of the special nature of $\b$ and could be applied to an arbitrary sequence $\a$, thus establishing directly that $\per(\a)$ always divides $\ord(C)$. Now observe that since $\b$ is $m$-periodic, the sequence of state vectors must also be $m$-periodic and so we have $t_{m + i} = t_i$ for all $i \geq 0$. But $t_{m+i} =  t_i C^m$, thus $t_i C^m = t_i I_k$ for all $i \geq 0$ and, in particular, this holds for all of the state vectors in the basis $\mathcal{B}$. It follows that $C^m = I_k$ and so the order $n$ must divide $m$. Since $m =  \per(\b)$ and $n = \ord(C)$ are positive integers that divide each other, we conclude that $\per(\b) = \ord(C)$ as desired.
\end{proof}

\begin{remark}\label{remark-max-period-proof}
Theorem~\ref{theorem-period-matrixorder} shows that the maximum period for a given recurrence with associated matrix $C$ is $\ord(C)$. The sequence $\b$ with initial sequence $t_0 = (0,\ldots,0,1)$ which is shown to attain this maximum in the proof is called the {\em impulse response sequence} for the given recurrence. If one examines the argument used at the end of the proof, one sees that any sequence $\a$, for which the first $k$ state vectors $\{ s_i \mid 0 \leq i \leq k-1 \}$ form a basis for $\mathbb{F}_q^k$, will also have the property that $\per(\a) = \ord(C)$.
\end{remark}

There is an important polynomial associated to the matrix $C$ and so also the recurrence~(\ref{eqn-lin-rec}). 
\begin{definition}
The {\em characteristic polynomial of $C \in GL_k(\mathbb{F}_q)$} is the polynomial 
\[ f(x) = \det(x I_k - C) \in \mathbb{F}_q[x]. \] 
Observe that if $C$ is a $k \times k$ matrix, then we have $\deg f(x) = k$.
\end{definition}
For matrices $C$ of the special form~(\ref{eqn-special-matrix}) above, it is a straightforward exercise using properties of determinants to show that 
\[   f(x) = \det(x I_k - C) = x^k - \sum_{i=0}^{k-1} c_i x^i. \]
The latter polynomial is very closely related to the original recurrence equation. Indeed, if one sets $n=0$ and formerly substitutes $x^i$ for $a_i$ in Equation~(\ref{eqn-lin-rec}), then this polynomial appears after moving all of the terms to one side.

We now show that there is a direct connection between the order of this polynomial (as defined at the beginning of the section) and the order of the matrix $C$.

\begin{theorem}\label{theorem-matrix-poly-order}
Let $C \in GL_k(\mathbb{F}_q)$ and let $f(x)$ be the associated characteristic polynomial. Then $\ord(C) = \ord(f(x))$.
\end{theorem}
\begin{proof} This result appears in~\cite[Lemma~8.26]{LN}. The proof there assumes some familiarity with advanced linear algebra so we give a self-contained proof here.

Consider the following set of polynomials 
\[ J = \{ g(x) \in \mathbb{F}_q[x] \mid g(C) = 0 \} \]
where the expression $g(C)$ is evaluated in the natural way (treating the constant term as a scalar multiple of the identity matrix) and the equation $g(C) = 0$ means $g(C)$ is the zero matrix. It is a consequence of the Cayley-Hamilton Theorem, that $f(C) = 0$ and so $f(x) \in J$. For matrices $C$ of the special form~(\ref{eqn-special-matrix}) above, this can also be seen directly as follows. 
If we consider the impulse response sequence $\b$ with $n$th state vector $t_n$, then since the state vectors also satisfy the recurrence~(\ref{eqn-lin-rec}), we have
\begin{eqnarray*}
 0 =  t_{n+k} - \sum_{i=0}^{k-1} c_i t_{n+i}  &=& t_n C^k - \sum_{i=0}^{k-1} c_i (t_n C^i) \\
&=& t_n \left( C^k - \sum_{i=0}^{k-1} c_i C^i \right) \\
&=& t_n f(C)
\end{eqnarray*}
for all $n \geq 0$. In particular, this holds for all of the vectors in the basis $\mathcal{B} = \{t_i \mid 0 \leq i \leq k-1 \}$ and so we must have $f(C) = 0$ as desired.

Having shown that $f(x) \in J$, we now observe that $J$ is an ideal and so $\langle f(x) \rangle \subseteq J$ where $\langle f(x) \rangle$ is the principal ideal generated by $f(x)$. In fact, equality must hold. This can be seen using the standard fact that $\mathbb{F}_q[x]$ is a principal ideal domain and verifying that $J$ does not contain any nonzero polynomials of smaller degree. Suppose that this were not the case and $g(x) \in J$ with $\deg g(x) < \deg f(x) = k$. Then since $g(C) = 0$, we would have $t_n g(C) = 0$ for all $n \geq 0$. Expanding the left-hand side when $n = 0$ would then give a dependence relation among the vectors in $\mathcal{B}$ which is a contradiction since these vectors are linearly independent. 

Having established that $J = \langle f(x) \rangle$, we finish by noting that $C^n = I_k$ holds if and only if $x^n - 1 \in J$ which in turn holds if and only if $f(x)$ divides $x^n - 1$. It follows from the definitions that $\ord(C) = \ord(f(x))$. Note that $f(0) = -c_0 \neq 0$, so $\gcd(x,f(x)) = 1$.
\end{proof}

Combining Theorem~\ref{theorem-period-matrixorder} and Theorem~\ref{theorem-matrix-poly-order}, we have
\begin{corollary}\label{corollary-period-poly-ord}
Let $f(x)$ be the characteristic polynomial of a linear recurrence and let $\a$ be any sequence satisfying the recurrence, then $\per(\a)$ divides $\ord(f(x))$. Moreover, there exist sequences with $\per(\a) = \ord(f(x))$.
\end{corollary}

\begin{corollary}\label{corollary-max-period-finite-field}
For each prime power $q$ and $k \in \Zpos$, there exists a linear recurrence sequence $\a$ over $\mathbb{F}_q$ of degree $k$  achieving the maximum possible period $\per(\a) = q^k - 1$.
\end{corollary}
\begin{proof} The multiplicative group of the finite field $\mathbb{F}_{q^k}$ is cyclic of order $q^k - 1$. Let $f(x)$ be the minimum polynomial of a generator for this group.  (Such an irreducible polynomial is said to be {\em primitive}). Then we have $\ord( f(x) ) = q^k - 1$ by Remark~\ref{remark-ord-irred} and so the statement now follows by Corollary~\ref{corollary-period-poly-ord}.
\end{proof}

More can be said about the possible values of $\per(\a)$. Observe that every periodic sequence $\a$ satisfies a linear recurrence. For instance, if $\a$ is $m$-periodic then it automatically satisfies the degree $m$ recurrence $a_{n+m} = a_n$ for all $n \geq 0$. Among the set of all such recurrences, we single out the one of smallest degree. This is uniquely determined since if there were two distinct recurrences of this degree, we could subtract, cancel and divide by the non-zero coefficient on the largest index term remaining to obtain a recurrence of even smaller degree satisfied by the sequence which would be a contradiction.

\begin{definition}
Let $\a$ be a nonzero periodic sequence. We define the {\em minimal polynomial $m(x)$} to be the characteristic polynomial of the (unique) linear recurrence of smallest degree satisfied by $\a$. For the zero sequence $\a = ( 0 )_{n=0}^\infty$, we set $m(x) = 1$.
\end{definition}

From the definition and preceding observations, we clearly have $\deg m(x) \leq \per(\a)$ and $\deg m(x) \leq \deg f(x)$ for all characteristic polynomials $f(x)$ of linear recurrences satisfied by $\a$. In fact, a much stronger statement holds.

\begin{theorem}\label{theorem-period-minpoly}
Let $\a$ be a periodic sequence and let $m(x)$ be the minimal polynomial of~$\a$. Then $f(x)$ is the characteristic polynomial of a linear recurrence satisfied by $\a$ if and only if $m(x)$ divides $f(x)$. We also have $\per(\a) = \ord(m(x))$.
\end{theorem}
\begin{proof}
For a proof of the divisibility statement, which we will not need in our subsequent work, see~\cite[Section~8.4]{LN}. Note that in~\cite{LN}, the minimal polynomial is defined via a divisibility condition. After existence and uniqueness have been demonstrated, it is straightforward to see that this is equivalent to the definition used here. 

For the second statement, observe that $\per(\a)$ divides $\ord(m(x))$ by Corollary~\ref{corollary-period-poly-ord} and, as noted above, we have $\ord(m(x)) \leq \per(\a)$. Equality follows immediately. Alternatively, one can see directly that if $\deg(m(x)) = k$, then the first $k$ state vectors of $\a$ must be linearly independent since a dependence relation would give rise to a recurrence of smaller degree satisfied by the full sequence of state vectors and hence also $\a$. It follows that the first $k$ state vectors form a basis of $\mathbb{F}_q^k$ since the dimension is $k$, and so we can apply Remark~\ref{remark-max-period-proof} to see that $\per(\a) = \ord(m(x))$.
\end{proof}

We conclude by showing how the results in this section provide information about the periods of sequences in the case of the Fibonacci recurrence.

\begin{example}
The Fibonacci recurrence $a_{n+2} = a_{n-1} + a_n$ has been extensively studied over both $\mathbb{Z}_n$ and finite fields. Recent work includes~\cite{GK, GRS} and additional references can be found in these papers.  
For this recurrence, the companion matrix is 
\[ C = 
\left(\begin{matrix}
0 & 1 \\
1 & 1 
\end{matrix}\right)
\]
and the characteristic polynomial is $f(x) = x^2 - x - 1$. The sequence $\a$ defined by this recurrence with initial state $s_0 = (0,1)$ is the impulse response sequence and so attains the maximum possible period $\per(\a) = \ord(f(x))$. Other initial states will give rise to sequences whose periods divide this maximum. Since the recurrence and initial state for $\a$ are both defined over the prime subfield $\mathbb{F}_p \subseteq \mathbb{F}_q$, the entire sequence will also lie in $\mathbb{F}_p$ and so it suffices to consider the case $q = p$ prime.

Determining the value of $\ord(f(x))$ involves first factoring $f(x)$ in $\mathbb{F}_p[x]$. We now consider various cases. If $p = 2$, then one observes that $f(x) = x^2 + x + 1$ is irreducible. Since $\alpha = x + \langle f(x) \rangle$ must have nontrivial multiplicative order in $\mathbb{F}_2[x]/\langle f(x) \rangle \cong \mathbb{F}_4$, we see that $\ord(f(x)) = 3$. Alternatively, one can simply generate the sequence and observe that the period is $3$ as we did in Example~\ref{example-fib-mod-2}.

If $p$ is odd, then $2$ is invertible in $\mathbb{F}_p$ and we can complete the square to write $f(x) = (x - 2^{-1})^2 - \Delta/2^{-2}$ where $\Delta = 5$ is the discriminant of $f(x)$. It follows that $f(x)$ will be irreducible over $\mathbb{F}_p$ if and only if $5$ is not a square in $\mathbb{F}_p$. By the law of quadratic reciprocity, we see that $5$ is not a square in $\mathbb{F}_p$ if and only if $p \equiv 2,3 \pmod 5$, and $5$ is a square if and only if $p \equiv 1,4 \pmod 5$ or $p = 5$.
\begin{itemize}
\item If $p \equiv 2,3 \pmod 5$, then $f(x)$ is irreducible over $\mathbb{F}_p$ and the order and hence period of the sequence will be the multiplicative order of $\alpha = x + \langle f(x) \rangle$ in $\mathbb{F}_p[x]/\langle f(x) \rangle \cong \mathbb{F}_{p^2}$. It follows that the order will be a divisor of $p^2 - 1$ that does not divide $p - 1$. (The second condition is forced since $\alpha$ cannot belong to the prime field $\mathbb{F}_p$ in this case.)
\item If $p \equiv 1,4 \pmod 5$, then $f(x)$ factors as a product of distinct linear factors over $\mathbb{F}_p$ and applying Theorem~\ref{theorem-ord-prod} we see that the period will divide $p - 1$. 
\item Finally, if $p = 5$, then we have a repeated linear factor $f(x) = (x - 3)^2$. Since $3$ has order $4$ in $\mathbb{F}_5$, we can apply Theorem~\ref{theorem-ord-irred-power} to see that $\ord(f(x)) = 20$. Alternatively, one can simply generate the sequence and observe that the period is $20$.
\end{itemize}

In fact, more can be said in the case where $f(x)$ is irreducible as shown in~\cite[Theorem~5]{GRS}. Given one root $\alpha$ of $f(x)$, the other root can be obtained by applying the Frobenius automorphism and is simply $\alpha^p$. Since the product of the roots of a quadratic is the constant coefficient, we obtain the relation $\alpha \cdot \alpha^p = \alpha^{p+1} = -c_0 = -1$. It follows that $\alpha^{2(p+1)} = 1$ and thus the order of $\alpha$ must divide $2(p+1)$ which is strictly smaller than $p^2 - 1$ for $p > 3$.
\end{example}

%%%%%%%%%%%%%%%%%%%%%%%%%%%%%%%%%%%%%%%%%%%%%%%%%%%%%%%%%%%%%%%%%%%%%%%%%%%%%%%%%%
%%%%%%%%%%%%%%%%%%%%%%%%%%%%%%%%%%%%%%%%%%%%%%%%%%%%%%%%%%%%%%%%%%%%%%%%%%%%%%%%%%
\section{Period sets of small degree recurrences}

In this section, we use the results from the previous section to determine the period sets for recurrences of small degree over finite fields. 

\begin{definition}
Let $k \in \Zpos$ and let $q$ be a prime power. We define the {\em period set of degree $k$ over $\mathbb{F}_q$} by
\[ P(k,\mathbb{F}_q) = \{ \per(\a) \mid \text{$\a$ satisfies a linear recurrence of degree $k$ over $\mathbb{F}_q$} \} 
\]
and the {\em order set of degree $k$ over $\mathbb{F}_q$} by
\[ O(k,\mathbb{F}_q) = \{ \ord(f(x)) \mid f(x) \in \mathbb{F}_q[x] \text{ and } \deg f(x) = k \}.  
\]
\end{definition}
A simple inductive argument, in which the inductive step involves adding together two copies of the recurrence equation for consecutive values of the leading index, shows that every sequence satisfying a recurrence of degree $j$ satisfies a recurrence of degree $k$ for all $k \geq j$. Hence, we have  $P(j,\mathbb{F}_q) \subseteq P(k,\mathbb{F}_q)$ for all $j \leq k$. Similarly, from the definition, we have $\ord(f(x)) = \ord(g(x))$ for $f(x) = x^d g(x)$ with $d \geq 0$, so then $O(j,\mathbb{F}_q) \subseteq O(k,\mathbb{F}_q)$ for all $j \leq k$.

The next lemma shows that if we want to compute the period set of a given degree, then it suffices to compute the corresponding order set.

\begin{lemma}\label{lemma-period-order}
For all $k \in \Zpos$ and prime powers $q$, we have $P(k,\mathbb{F}_q) = O(k,\mathbb{F}_q)$.
\end{lemma}
\begin{proof}
We check containment in both directions. Let $n \in P(k,\mathbb{F}_q)$. Then $n = \per(\a)$ for some sequence $\a$ defined by a linear recurrence of degree $k$ over $\mathbb{F}_q$. Let $f(x)$ be the associated characteristic polynomial and let $m(x)$ be the minimal polynomial of~$\a$. Let $d = \deg m(x)$. Then $d \leq k$ by definition of the minimal polynomial. It follows from Theorem~\ref{theorem-period-minpoly} that
\[  n = \per(\a) = \ord(m(x)) \in  O(d,\mathbb{F}_q) \subseteq O(k,\mathbb{F}_q). \]
Thus $P(k,\mathbb{F}_q) \subseteq O(k,\mathbb{F}_q)$.

Now suppose $n \in O(k,\mathbb{F}_q)$. Then $n = \ord(f(x))$ for some $f(x) \in \mathbb{F}_q[x]$ with $\deg f(x) = k$. 
By Corollary~\ref{corollary-period-poly-ord}, there exists a sequence $\a$  satisfying the linear recurrence associated to $f(x)$ with $\per(\a) = \ord(f(x)) = n$, and so $n \in P(k,\mathbb{F}_q)$. Thus $O(k,\mathbb{F}_q) \subseteq P(k,\mathbb{F}_q)$. The desired set equality now follows.
\end{proof}

Before using Lemma~\ref{lemma-period-order} to give an explicit description of $P(k,\mathbb{F}_q)$ for small values of $k$, we introduce some more notation. Given $n \in \Zpos$, we let $D(n)$ denote the set of all positive integer divisors of $n$. For example, $D(6) = \{1, 2, 3, 6 \}$. Given an integer $a$ and subsets $S_1, S_2 \subseteq \mathbb{Z}$, we define $a S_1 = \{ a x \mid x \in S_1 \}$ and $S_1 S_2 = \{ x y \mid x \in S_1, y \in S_2 \}$. For example, $5 D(6) = \{5, 10, 15, 30\}$ and $D(2) D(6) = \{ 1, 2, 3, 4, 6, 12 \}$.

\begin{theorem}\label{theorem-main-contain}
Let $\mathbb{F}_q$ have prime characteristic $p$. For all $k \in \Zpos$, we have
\[  \bigcup_{i = 1}^k \{ p^j  \mid 0 \leq j \leq t_i\} D(q^i - 1)  \: \subseteq \: P(k, \mathbb{F}_q) \]
where $t_i = \lceil \log_p \lfloor k/i \rfloor \rceil = \min\{t \in \mathbb{Z} \mid p^t \geq \lfloor k/i \rfloor \}$ for $1 \leq i \leq k$.
\end{theorem}
\begin{proof}
To prove the containment, let $1 \leq i \leq k$ and consider $n = p^j d$ with $0 \leq j \leq t_i$ and $d \in D(q^i - 1)$. We show that $n \in P(k, \mathbb{F}_q)$ by constructing an appropriate polynomial $f(x)$ with $\ord(f(x)) = n$. 

Since the multiplicative group of $K = \mathbb{F}_{q^i}$ is cyclic of order $q^i - 1$ and $d$ divides $q^i - 1$, we can find an element $\alpha \in K$ of order $d$. Let $g(x) \in \mathbb{F}_q[x]$ be the minimal polynomial of $\alpha$ over $\mathbb{F}_q$. Then $g(x)$ is irreducible with $\deg g(x) \leq i$ and $\ord(g(x)) = d$ by Remark~\ref{remark-ord-irred}.
Now set $f(x) = g(x)^b$ with $b = 1$ if $j = 0$ and $b = p^{j-1} + 1$ if $j \geq 1$. Note that in both cases, $j$ is the smallest integer such that $p^j \geq b$. Then $\ord(f(x)) = \ord(g(x)) p^j = d p^j = n$ by Theorem~\ref{theorem-ord-irred-power}, and 
\[ \deg f(x) = b \deg g(x) \leq  (p^{j-1} + 1) i \leq (p^{t_i - 1} + 1) i \leq \lfloor k/i \rfloor i \leq k.
\]
The inequality $p^{t_i - 1} + 1 \leq \lfloor k/i \rfloor$ used in the second-last step above follows directly from the definition of $t_i$. We conclude that $n \in O(k,\mathbb{F}_q) = P(k,\mathbb{F}_q)$ by Lemma~\ref{lemma-period-order}.
\end{proof}

\begin{remark}
The exponent bounds $t_i$ and corresponding sets of prime powers appearing in Theorem~\ref{theorem-main-contain} are often small. Observe that $t_i \geq 0$ for $1 \leq i \leq k$ and $t_i = 0$ occurs for $\lfloor k/2 \rfloor + 1 \leq i \leq k$. Since the largest value of the floor function $\lfloor k/i \rfloor$ occurs when $i = 1$, it follows that for $p > k$, we have $\{ p^j  \mid 0 \leq j \leq t_i\}  = \{1, p \}$ for $1 \leq i \leq \lfloor k/2 \rfloor$. Larger sets of prime powers can occur when $k$ is large relative to $p$.
\end{remark}

We now come to our main result.

\begin{theorem}  \label{theorem-period-sets}
Let $\mathbb{F}_q$ have prime characteristic $p$. For $1 \leq k \leq 4$, the set containment in Theorem~\ref{theorem-main-contain} is actually an equality. In particular, we have
\begin{itemize}
\item[(i)] $P(1, \mathbb{F}_q) = D(q - 1)$.
\item[(ii)] $P(2, \mathbb{F}_q) = D(q^2 - 1) \cup p D(q-1)$.
\item[(iii)] $\displaystyle P(3, \mathbb{F}_q) =
\begin{cases}
\displaystyle  D(q^3 - 1) \cup D(q^2 - 1) \cup p D(q - 1), &\text{if $p \geq 3$.}\\[15 pt]
\displaystyle  D(q^3 - 1) \cup D(q^2 - 1) \cup \{2, 4\} D(q - 1), &\text{if $p = 2$.}
\end{cases}$ \\[5 pt]
\item[(iv)] $\displaystyle P(4, \mathbb{F}_q) = 
\begin{cases}
\displaystyle  D(q^4 - 1) \cup D(q^3 - 1) \cup p  D(q^2 - 1), &\text{if $p \geq 5$.}\\[15 pt]
\displaystyle  D(q^4 - 1) \cup D(q^3 - 1) \cup p D(q^2 - 1) \cup p^2 D(q - 1), &\text{if $p = 2, 3$.}
\end{cases}$
\end{itemize}
\end{theorem}

\begin{proof}
It is straightforward to verify that the set expressions appearing on the right above match the one appearing in  Theorem~\ref{theorem-main-contain} for
$1 \leq k \leq 4$. Note that some simplification has been carried out. For instance, if $i$ divides $j$, then $q^i - 1$ divides $q^j - 1$ and one then has $D(q^i - 1) \subseteq D(q^j - 1)$. It follows that if both of these occur in a set union then the first is redundant and can be omitted. In particular, $D(q - 1)$ has been omitted from most of the expressions above. $D(q^2 - 1)$ has also been omitted when $D(q^4 - 1)$ occurs.

Since  Theorem~\ref{theorem-main-contain} shows containment in one direction, we will establish equality by checking the reverse containment for each value $1 \leq k \leq 4$. We have $P(k, \mathbb{F}_q) = O(k,\mathbb{F}_q)$ by Lemma~\ref{lemma-period-order}, so it will suffice to verify that $\ord(f(x))$ belongs to the given set union for all $f(x)$ with $\deg f(x) = k$. We will make frequent use of the fact that if $g(x) \in \mathbb{F}_q[x]$ is irreducible and  $\deg g(x) = d$, then $\ord(g(x)) \in D(q^d - 1)$ which follows from Remark~\ref{remark-ord-irred}. We also note that the order of a polynomial is invariant under multiplication by nonzero scalars. Thus we can restrict our attention to monic polynomials and will also assume, without loss of generality, that all factors are monic when working with factorizations of such polynomials.

If $\deg f(x) = 1$, then $f(x)$ must be irreducible  so $\ord(f(x)) \in D(q-1)$ as noted above. This establishes part~(i). 

If $\deg f(x) = 2$, then either $f(x)$ is irreducible, in which case we have $\ord(f(x)) \in D(q^2 - 1)$, or we have a factorization $f(x) = g_1(x) g_2(x)$ with $\deg g_1(x) = \deg g_2(x) = 1$. There are then two cases to consider. If the factors are distinct (not associates), then we see that $\ord(f(x))$ is the least common multiple of $\ord(g_1(x))$ and $\ord(g_2(x))$ by Theorem~\ref{theorem-ord-prod}. Since both orders divide $q-1$, it follows that $\ord(f(x)) \in D(q-1)$. On the other hand, if $g_1(x)$ and $g_2(x)$ are associates, then they must be equal since we are assuming all factors are monic. Thus we have $f(x) = g_1(x)^2$. Applying Theorem~\ref{theorem-ord-irred-power}, we see that $\ord(f(x)) = p \, \ord(g_1(x))
\in p D(q - 1)$ since $p = p^1 \geq 2$ for all $p$. Combining the above cases, we have thus established that if $\deg f(x) = 2$, then
\[ \ord(f(x)) \in D(q^2 - 1)  \cup D(q-1) \cup p D(q-1)  = D(q^2 - 1) \cup p D(q-1).  \]
This establishes part~(ii).

The same sort of arguments are used to handle the remaining values $k = 3,4$ and we now outline the main subcases. If $\deg f(x) = 3$, then one of the following must hold:
\begin{itemize}
\item $f(x)$ is irreducible. In this case, we have $\ord(f(x)) \in D(q^3 - 1)$.
\item $f(x) = g(x) h(x)$ with $g(x)$, $h(x) \in \mathbb{F}_q[x]$ irreducible and $\deg g(x) = 2$ and $\deg h(x) = 1$. Then $\ord(g(x)) \in D(q^2 - 1)$ and $\ord(h(x)) \in D(q - 1)$. Since $q - 1$ divides $q^2 - 1$, we see that $\ord(f(x)) \in D(q^2 - 1)$.
\item $f(x) = g_1(x) g_2(x) g_3(x)$ with $\deg g_1(x) = \deg g_2(x) = \deg g_3(x) = 1$ and the factors pairwise distinct. We see that $\ord(g_i(x)) \in D(q - 1)$ for $i = 1, 2, 3$. Since $D(q-1)$ is closed under least common multiples, we have $\ord(f(x)) \in D(q - 1)$.
\item $f(x) = g_1(x)^2 g_2(x)$ with $\deg g_1(x) = \deg g_2(x) = 1$ and $g_1(x)$ distinct from $g_2(x)$. We see that $\ord(g_1(x)^2) \in p D(q-1)$ and $\ord(g_2(x)) \in D(q-1)$. Since $D(q-1)$ is closed under least common multiples, we have $\ord(f(x)) \in p D(q-1)$.
\item $f(x) = g(x)^3$ with $\deg g(x) = 1$. We have $\ord(g(x)) \in D(q-1)$. The smallest value of $t$ such that $p^t \geq 3$ is $t = 1$ if $p \geq 3$ and $t = 2$ if $p = 2$. It follows that $\ord(f(x)) \in p D(q-1)$ if $p \geq 3$, and $\ord(f(x)) \in 4 D(q - 1)$ if $p = 2$.
\end{itemize}
Combining the above, we see that if $\deg f(x) = 3$ and $p \geq 3$, then
\begin{eqnarray*}
 \ord(f(x)) &\in& D(q^3 - 1) \cup D(q^2 - 1) \cup D(q - 1) \cup p D(q - 1) \\
 &=& D(q^3 - 1) \cup D(q^2 - 1) \cup p D(q - 1) 
\end{eqnarray*}
If $p = 2$, then $4 D(q-1)$ must also be included in the union. This establishes part~(iii).

Finally, if $\deg f(x) = 4$, then one of the following must hold:
\begin{itemize}
\item $f(x)$ is irreducible. Then $\ord(f(x)) \in D(q^4 - 1)$.
\item $f(x)$ is the product of degree $3$ and degree $1$ irreducible polynomials. Since $q - 1$ divides $q^3 - 1$, we see that $\ord(f(x)) \in D(q^3 - 1)$.
\item $f(x)$ is the product of two irreducible quadratics. One must consider both the case where the quadratic factors are distinct and also the case where one is repeated. Combining, one sees that $\ord(f(x)) \in \{1, p\} D(q^2 - 1)$.
\item $f(x)$ is the product of an irreducible quadratic and two degree $1$ polynomials (where the latter might be repeated). Then $\ord(f(x)) \in
\{1, p \} D(q^2 - 1)$. In deriving this statement, note that the least common multiple of an element in $D(q^2 - 1)$ and $p D(q - 1)$ will lie in $p D(q^2 - 1)$. 
\item $f(x)$ is the product of $4$ linear factors. These could all be distinct or there could be some repetition. The cases where $f(x)$ includes a repeated factor $g(x)^3$ or $g(x)^4$ are the ones where the behavior is slightly different for small $p$. In particular, the smallest value of $t$ such that $p^t \geq 4$ is $t = 1$ for $p \geq 5$ and $t = 2$ for $p = 2, 3$. Analyzing all the cases, we see that if $p \geq 5$, then $\ord(f(x)) \in \{1, p\} D(q-1)$, and if $p = 2,3$, then 
$\ord(f(x)) \in \{1, p, p^2\} D(q-1)$.
\end{itemize}
Combining the above, we see that if $\deg f(x) = 4$ and $p \geq 5$, then
\begin{eqnarray*}
 \ord(f(x)) &\in& D(q^4 - 1) \cup D(q^3 - 1) \cup \{1, p\} D(q^2 - 1) \cup \{1, p\} D(q - 1) \\
 &=& D(q^4 - 1) \cup D(q^3 - 1) \cup p  D(q^2 - 1).
\end{eqnarray*}
If $p = 2,3$, then $p^2 D(q-1)$ must also be included in the union. This establishes part~(iv) and completes the proof.
\end{proof}

\begin{example} Consider the field $\mathbb{F}_2$ so $q = p = 2$. Using Theorem~\ref{theorem-period-sets}, we can easily compute all possible periods for linear recurrence sequences of degree $k \leq 4$ over this field. We have:
\begin{align*}
P(1, \mathbb{F}_2) &= D(1) = \{ 1 \}. \\
P(2, \mathbb{F}_2) &= D(3) \cup 2 D(1) = \{1, 2, 3 \}. \\
P(3, \mathbb{F}_2) &= D(7) \cup D(3) \cup \{2, 4\} D(1) = \{1, 2, 3, 4, 7 \}. \\
P(4, \mathbb{F}_2) &= D(15) \cup D(7) \cup 2 D(3) \cup 4 D(1) = \{1, 2, 3, 4, 5, 6, 7, 15 \}.
\end{align*}
\end{example}

\begin{remark}
It is relatively easy to find examples showing that the containment in Theorem~\ref{theorem-main-contain} can be strict once $k \geq 5$. For example, the polynomials $g(x) = x^2 + x + 1$ and $h(x) = x^3 + x + 1$ are primitive over $\mathbb{F}_2$ and so achieve the maximum possible order $2^k - 1$ relative to their degree $k$ as discussed in Corollary~\ref{corollary-max-period-finite-field}. If we set $f(x) = g(x) h(x) = x^5 + x^4 + 1 = x^5 - x^4 - 1$, then $\ord(f(x)) = \lcm(\ord(g(x)), \ord(h(x))) = \lcm(3, 7) = 21$. It follows that the impulse response sequence defined by the corresponding linear recurrence 
\[ a_5 = a_4 + a_0 \]
has period $21$. This can also be verified directly by computing the sequence
\[ \underline{0, 0, 0, 0, 1}, 1, 1, 1, 1, 0, 1, 0, 1, 0, 0, 1, 1, 0, 0, 0, 1, \underline{0, 0, 0, 0, 1}, \ldots \]
Observe that $21 \notin D(2^k - 1)$ for $1 \leq k \leq 5$ which demonstrates that the containment in Theorem~\ref{theorem-main-contain} is strict when $k = 5$ and $q = p = 2$. 
\end{remark}

%%%%%%%%%%%%%%%%%%%%%%%%%%%%%%%%%%%%%%%%%%%%%%%%%%%%%%%%%%%%%%%%%%%%%%%%%%%%%%%%%%
%%%%%%%%%%%%%%%%%%%%%%%%%%%%%%%%%%%%%%%%%%%%%%%%%%%%%%%%%%%%%%%%%%%%%%%%%%%%%%%%%%
\section{Linear recurrences and period sets over other finite rings}

Let $q$ be a power of the prime $p$. The finite field $\mathbb{F}_q$ can be constructed as a quotient $\mathbb{F}_p[t]/\langle f(t) \rangle$ where $f(t) \in \mathbb{F}_p[t]$ is irreducible. In this section, we consider sequences defined by linear recurrences over a broader class of finite commutative rings, some of which arise by weakening this assumption on $f(t)$. 

Let $f(t) = \prod_{i=1}^r f_i(t)$ with $f_i(t) \in \mathbb{F}_p[t]$ monic and  irreducible for $1 \leq i \leq r$.  Define $\mathcal{R} = \mathbb{F}_p[t]/\langle f(t) \rangle$. 
If we assume that the factors $f_i(t)$ are pairwise distinct (hence relatively prime), then we can apply the Chinese Remainder Theorem to obtain a ring isomorphism
\[  \mathcal{R} \cong \mathbb{F}_p / \langle f_1(t) \rangle \oplus \ldots \oplus \mathbb{F}_p / \langle f_r(t) \rangle \cong \mathbb{F}_{p^{d_1}} \oplus \ldots \oplus \mathbb{F}_{p^{d_r}} 
\]
where $d_i = \deg f_i(t)$ for $1 \leq i \leq r$. The induced projection $\proj_i$ from $\mathcal{R}$ onto the $i$th component $\mathbb{F}_p / \langle f_i(t) \rangle \cong \mathbb{F}_{p^{d_i}}$ is given explicitly by 
\[   g(t) + \langle f(t) \rangle  \longmapsto  g(t) + \langle f_i(t) \rangle. \]
More generally, we will consider rings of the form $\mathcal{R} = \mathbb{F}_{q_1} \oplus \ldots \oplus \mathbb{F}_{q_r}$ in which the finite fields in different components are not required to have the same characteristic.

Suppose we have a sequence $\a = (a_j)_{j=1}^\infty$ defined over such a ring $\mathcal{R}$. Applying the projection maps $\proj_i: \mathcal{R} \rightarrow \mathbb{F}_{q_i}$, we obtain $r$ sequences $\proj_i(\a) := (\proj_i(a_j))_{j=1}^\infty$, each one defined over the corresponding component field $\mathbb{F}_{q_i}$. Conversely, given sequences $(a_j^{(i)})_{j=1}^\infty$ defined over $\mathbb{F}_{q_i}$ for $1 \leq i \leq r$, we can form $r$-tuples to obtain a sequence $\a = (a_j^{(1)}, \ldots, a_j^{(r)})_{j=1}^\infty$ defined over $\mathcal{R}$. The following lemma is easily verified.

\begin{lemma}
The sequence $\a$ is periodic if and only if $\proj_i(\a)$ is periodic for all $1 \leq i \leq r$. When $\a$ is periodic, we have
\[  \per(\a) = \lcm(\per(\proj_1(\a)), \ldots, \per(\proj_r(\a))). \]
\end{lemma}

Some straightforward algebraic manipulations show that the sequence $\a$ is defined by a linear recurrence of degree~$k$ over $\mathcal{R}$ if and only if this holds for all of the component sequences $\proj_i(\a)$. The recurrence equations for the latter are obtained by simply applying the projection maps to the coefficients of the recurrence equation for~$\a$. 

We will continue to restrict attention to linear recurrence equations in which the lowest indexed coefficient $c_0$ is a unit. This condition holds for $c_0 \in \mathcal{R}$ if and only if $\proj_i(c_0)$ is a unit in $\mathbb{F}_{q_i}$ for all $i$, equivalently $\proj_i(c_0) \neq 0$ for all $i$. Recall from Section~\ref{section-introduction} that sequences defined by such a linear recurrence of degree $k$ over $\mathcal{R}$ will be periodic with $\per(\a)$ bounded above by $|\mathcal{R}|^k - 1$. As we will see shortly, this upper bound on the period is not always attained.

Extending the period set notation introduced in the previous section and using the lemma and other observations above, we have the following result.
\begin{lemma}\label{corollary-algebra-periods}
Let $\mathcal{R} = \mathbb{F}_{q_1} \oplus \ldots \oplus \mathbb{F}_{q_r}$ and  $k \geq 1$. Then
\[ P(k, \mathcal{R}) = \{ \lcm(\omega_1,\ldots,\omega_r) \mid \omega_i \in P(k,\mathbb{F}_{q_i}) \: \mathrm{for} \: 1 \leq i \leq r \}. \]
\end{lemma}

\begin{example}
Let $\mathcal{R} = \mathbb{F}_2 \oplus \mathbb{F}_3 \oplus \mathbb{F}_5$. First consider $k = 1$. Using Theorem~\ref{theorem-period-sets}, we see that:
\begin{align*}
P(1, \mathbb{F}_2) &= D(1) = \{ 1 \}. \\
P(1, \mathbb{F}_3) &= D(2) = \{ 1, 2 \}. \\
P(1, \mathbb{F}_5) &= D(4) = \{ 1, 2, 4 \}. 
\end{align*}
It follows by Lemma~\ref{corollary-algebra-periods} that
\begin{align*}
P(1, \mathcal{R}) &= \{ \lcm(\omega_1,\omega_2, \omega_3) \mid \omega_1 \in P(1, \mathbb{F}_2), \omega_2 \in P(1, \mathbb{F}_3), \omega_3 \in P(1, \mathbb{F}_5) \} \\
&= \{1, 2, 4 \}.
\end{align*}
When $k = 2$, 
\begin{align*}
P(2, \mathbb{F}_2) &= D(3) \cup 2D(1) = \{ 1, 2, 3 \}. \\
P(2, \mathbb{F}_3) &= D(8) \cup 3D(2) = \{ 1, 2, 3, 4, 6, 8 \}. \\
P(2, \mathbb{F}_5) &= D(24) \cup 5D(4) = \{ 1, 2, 3, 4, 5, 6, 8, 10, 12, 20, 24 \}. 
\end{align*}
By Lemma~\ref{corollary-algebra-periods}, we then have
\begin{align*}
P(2, \mathcal{R}) &= \{ \lcm(\omega_1,\omega_2, \omega_3) \mid \omega_1 \in P(2, \mathbb{F}_2), \omega_2 \in P(2, \mathbb{F}_3), \omega_3 \in P(2, \mathbb{F}_5) \} \\
%&= \{2^a 3^b 5^c \mid 0 \leq a \leq 3, \, 0 \leq b \leq 1, \, 0 \leq c \leq 1 \}.
&= \{ 1, 2, 3, 4, 5, 6, 8, 10, 12, 15, 20, 24, 30, 40, 60, 120 \}.
\end{align*}
We will skip over $k = 3$ and $k = 4$ since our main point is to demonstrate how easily one can obtain the full period sets without doing a brute force enumeration of all possible combinations of linear recurrence equations and initial states. Even for $k = 2$, there are $8 \cdot 30 = 240$ choices for the coefficients $(c_0,c_1)$ with $c_0$ a unit, and each of these can be paired with $30^2 = 900$ different initial states $s_0 = (a_0, a_1)$.
\end{example}

\begin{remark}
By the Chinese Remainder Theorem, we have $\mathbb{F}_2 \oplus \mathbb{F}_3 \oplus \mathbb{F}_5 \cong \mathbb{Z}_{30}$ so the calculations in the preceding example can be viewed as determining the period sets for linear recurrences of small degree defined over $\mathbb{Z}_{30}$. More generally, the methods in this paper can be applied to find period sets of linear recurrences defined over $\mathbb{Z}_n$ for any $n$ that decomposes as a product of distinct prime numbers. The periods of sequences in the case where $n$ includes nontrivial prime powers have also been investigated. We refer the reader to~\cite{P, W} for more details.
\end{remark}

We have the following upper bound on the periods of sequences defined over such rings.

\begin{lemma}\label{corollary-algebra-max-period}
Let $\mathcal{R} = \mathbb{F}_{q_1} \oplus \ldots \oplus \mathbb{F}_{q_r}$ and $k \geq 1$. Let $\a$ be a sequence defined by a linear recurrence of degree $k$ over $\mathcal{R}$. Then 
\[ \per(\a) \leq \prod_{i=1}^r (q_i^k - 1).\]
\end{lemma}
\begin{proof}
We have $\per(\a) = \lcm(\omega_1,\ldots,\omega_r) \leq \prod_{i=1}^r \omega_i$ with  $\omega_i \in P(k,\mathbb{F}_{q_i})$ for $1 \leq i \leq r$. Since $\omega_i \leq q_i^k - 1$ for all $i$, the result follows.
\end{proof}

One easy consequence of this upper bound is the following characterization of the fields among such rings. Note that $\mathcal{R} = \mathbb{F}_{q_1} \oplus \ldots \oplus \mathbb{F}_{q_r}$ is not a field if $r > 1$ due to the presence of zero divisors.

\begin{theorem}
Let $\mathcal{R} = \mathbb{F}_{q_1} \oplus \ldots \oplus \mathbb{F}_{q_r}$ and $k \geq 1$. Among all sequences $\a$ defined by linear recurrences of degree $k$ over $\mathcal{R}$, the maximum period $|\mathcal{R}|^k - 1$ is achieved if and only if  $\mathcal{R}$ is a field ($r = 1$).
\end{theorem}
\begin{proof}
By Corollary~\ref{corollary-max-period-finite-field}, the maximum can be achieved when $\mathcal{R}$ is a field so the reverse implication holds. The forward implication follows from the preceding lemma by observing that if $r > 1$, then 
\[ \per(\a) \leq \prod_{i=1}^r (q_i^k - 1) <  \left(\prod_{i=1}^r q_i^k  \right) - 1 = |\mathcal{R}|^k - 1 \]
for all sequences $\a$ defined by a linear recurrence of degree~$k$ over $\mathcal{R}$ .
\end{proof}

We have investigated the maximum periods for some families of rings that are not fields.
One particularly simple family are those of the form $\mathcal{A}_n = \mathbb{F}_p[t]/\langle t^n - 1 \rangle$ for $n > 1$. These are also called {\em cyclic group algebras} since the elements of $\mathcal{A}_n$ can be identified with formal linear combinations of the elements of a cyclic group of order $n$ with coefficients in $\mathbb{F}_p$. The multiplication operation can then be viewed as arising by extending the group multiplication to such linear expressions in a natural way.

The factorization $t^n - 1 = (t - 1)(t^{n-1} + \ldots + t + 1)$ immediately implies that $\mathcal{A}_n$ is not a field for $n > 1$ and the maximum period for a linear recurrence sequence  of degree $k$ over $\mathcal{A}_n$ must be strictly smaller than $|\mathcal{A}_n|^k - 1 = p^{nk} - 1$. (Note that the latter assertion holds even when $f(t)$ has repeated irreducible factors which occurs when $p$ divides $n$.) Computer experiments show that there is a lot of variation in the maximum period. This is not surprising since it  depends on the structure of $\mathcal{A}_n$ which in turn depends on the factorization of $t^n - 1$ into irreducibles over $\mathbb{F}_p$ as $p$ varies.
In general, it seems that the larger the number of factors of $t^n - 1$, the smaller the maximum period, although we have not formulated a precise statement regarding this relationship. 

We conclude with one general observation about the family $\{ \mathcal{A}_n \}_{n=1}^\infty$. It is well known that if $n = \ell$ is a prime integer, then the cyclotomic polynomial $\Phi_\ell(t) = t^{\ell-1} + \ldots + t + 1$ is irreducible over $\mathbb{Q}$. It follows from the Frobenius Density Theorem (see~\cite[pg 32]{SL} for more details on the latter), that $\Phi_\ell(t)$  remains irreducible when reduced modulo $p$ for infinitely many primes $p$. For such $p$, we have $\mathcal{A}_\ell \cong \mathbb{F}_p \oplus \mathbb{F}_{p^{\ell -1}}$ and so the maximum period of linear recurrence sequences of degree $k$ over $\mathcal{A}_\ell$ is at least $p^{(\ell -1)k} - 1 \approx \frac{1}{p^k}(|\mathcal{A}_\ell|^k - 1)$ in this case. This follows since there are sequences defined over the field $ \mathbb{F}_{p^{\ell -1}}$ in the second component with period $p^{(\ell -1)k} - 1$ by Corollary~\ref{corollary-max-period-finite-field}.

\end{document}